\documentclass[letterpaper,11pt]{article}
\usepackage{amsmath}
\usepackage{amssymb}
\usepackage{color}
\usepackage{mathtools}
\usepackage{amsthm, amscd, mathrsfs}
\usepackage[top=1in, bottom=1.7in, left=1.2in, right=1.2in]{geometry}

\newtheorem{thm}{Theorem}[section]
\newtheorem{prop}[thm]{Proposition}
\newtheorem{lem}[thm]{Lemma}

\theoremstyle{remark}

\newtheorem{example}[thm]{\bf Example}

\newtheorem{defn}[thm]{\bf Definition}

\newcommand{\mbb}{\mathbb}

\newcommand{\N}{\mathbb{N}}
\newcommand{\Z}{\mbb{Z}}
\newcommand{\Q}{\mbb{Q}}

\newcommand{\eps}{\varepsilon}

\newcommand{\Set}[2]{\left\{\, #1 \;\middle|\; #2 \,\right\}}
\newcommand{\oneset}[1]{\left\{\, #1 \,\right\}}

\newcommand{\floor}[1]{\left\lfloor\mathinner{#1} \right\rfloor}

\newcommand{\gen}[1]{\left< \mathinner{#1} \right>}
\newcommand{\genr}[2]{\left< \, \mathinner{#1}\; \middle|\;\mathinner{#2} \, \right>}
\newcommand{\abs}[1]{\left|\mathinner{#1}\right|}

\newcommand{\smalloverline}[1]
{{\mspace{.8mu}\overline{\mspace{-.8mu}#1\mspace{-.8mu}}\mspace{.8mu}}}
\newcommand{\ov}[1]{\smalloverline{#1}}

\newcommand{\rank}{\mathrm{rk}}
\newcommand{\Oh}{\mathcal{O}}

\newcommand{\UTZ}[1]{UT_{#1}(\Z)}

\usepackage{fancyhdr}

 \newcounter{AlgorithmusCounter}[section]

\renewcommand{\theAlgorithmusCounter}{\arabic{section}.\arabic{AlgorithmusCounter}}
\newenvironment{algorithm}[1]{%
  \refstepcounter{AlgorithmusCounter}%
  \begin{tabbing}
    \hspace{8.5cm}\=\kill
    \tab\rule{\textwidth}{0.8pt} \\
    \parbox{.99\textwidth}{\textbf{Algorithm \theAlgorithmusCounter\ }
    \emph{#1}} \\[-2.5mm]
    \tab\rule{\textwidth}{0.8pt} 
}{
    \\[-3mm]
    \tab\rule{\textwidth}{0.8pt}
  \end{tabbing}
}

\newcommand{\comment}[1]{ \> \tab $(*$ {\footnotesize #1} $*)$}  

\newcommand{\coprocedure}{\textbf{procedure }}

\newcommand{\cobegin}{\textbf{begin }}
\newcommand{\coend}{\textbf{end }}

\newcommand{\cofor}{\textbf{for }}

\newcommand{\coforeach}{\textbf{for each }}
\newcommand{\codo}{\textbf{do }}

\newcommand{\coendfor}{\textbf{endfor }}
\newcommand{\coreturn}{\textbf{return }}

\newcommand{\codownto}{\textbf{downto }}

\newcommand{\corepeat}{\textbf{repeat }}
\newcommand{\countil}{\textbf{until }}

\newcommand{\IndentLength}{\hspace*{1.5em}}
\newcommand{\tab}{}
\newcommand{\tabb}{\IndentLength}
\newcommand{\tabbb}{\tabb\IndentLength}
\newcommand{\tabbbb}{\tabbb\IndentLength}
\newcommand{\tabbbbb}{\tabbbb\IndentLength}

\begin{document}

\title{On the dimension of matrix embeddings of torsion-free nilpotent groups}
\author{Funda Gul, Armin Wei\ss}
\date{\today}
\maketitle
\begin{abstract}
    Since the work of Jennings (1955), it is well-known that any finitely generated torsion-free nilpotent group can be embedded into unitriangular integer matrices $\UTZ{N}$ for some $N$. In 2006, Nickel proposed an algorithm to calculate such embeddings.
	In this work, we show that if $\UTZ{n}$ is embedded into $\UTZ{N}$ using Nickel's algorithm, then $N\geq 2^{n/2 -2}$ if the standard ordering of the Mal'cev basis (as in Nickel's original paper) is used. In particular, we establish an exponential worst-case running time of Nickel's algorithm.  
	
    On the other hand, we also prove a general exponential upper bound on the dimension of the embedding by showing that for any torsion free, finitely generated nilpotent group the matrix representation produced by Nickel's algorithm has never larger dimension than Jennings' embedding. Moreover, when starting with a special Mal'cev basis, Nickel's embedding for $\UTZ{n}$ has only quadratic size.   
	Finally, we consider some special cases like free nilpotent groups and Heisenberg groups and compare the sizes of the embeddings.
\smallskip

\end{abstract}

	\section{Introduction}
    A classical result due to Jennings \cite{Jennings55} shows that every finitely generated torsion-free nilpotent group ($\tau$-group) can be embedded into some group of unitriangular matrices over the integers. Embeddings into matrix groups are desirable for various reasons: they allow to apply the powerful tool of linear algebra to prove new results about the groups; moreover, many computations can be performed efficiently with matrices~-- in particular, the word problem for linear groups can be solved in logarithmic space \cite{lz77}. Embeddings of nilpotent groups are the basic building blocks for embeddings of polycyclic groups (see e.\,g.\ \cite{LoO99}), which are of particular interest because of their possible application in non-commutative cryptography \cite{EickK04}. For instance, in \cite{MyasnikovR15}, matrix embeddings were used to break such a cryptosystem based on the conjugacy problem in a certain class of polycyclic groups.
    
    Since Jennings' embedding (1955), several other descriptions of such embeddings \cite{Hall57, KargapolovM79} have been given and also algorithms \cite{LoO99, GraafN02, Nickel06} for computing such embeddings from a given Mal'cev presentation. 
    The presumably most efficient of these algorithms is due to Nickel \cite{Nickel06}: it uses the multiplication polynomials associated to the Mal'cev basis in order to compute a $G$-submodule of the dual space of the group algebra $\Q G$~-- an approach similar to the description of the embedding in \cite[Section 17.2]{KargapolovM79}. The multiplication polynomials were first described by Hall~\cite{Hall57}; they can be computed with the \emph{Deep Thought} algorithm~\cite{LeedhamGreenS98}. 
    
    Up to now there are no bounds known neither on the running time nor on the dimension of the embedding obtained by Nickel's algorithm. 
    In \cite{HabeebK13}, a polynomial bound for both is claimed; however, there is a gap in the proof. Indeed, here we prove these results to be wrong: our main result (Theorem~\ref{thm:complexity}) establishes the lower bound of $N\geq 2^{n/2 -2}$ for the embedding of $\UTZ{n}$ (with the standard Mal'cev basis as in Nickel's paper) into $\UTZ{N}$ computed by Nickel's algorithm. In particular, we establish an exponential blow-up for the dimension, which also implies an exponential running time since the output has to be written down. 
	On the other hand, we show the upper bound $N \leq 3^n$. Moreover, by reordering the Mal'cev basis of $\UTZ{n}$, Nickel's algorithm produces an embedding of size only $\Oh(n^2)$. Our exponential lower bounds also imply that for breaking cryptographic systems based on polycyclic groups the usage of matrix embeddings (at least with the known algorithms to compute them) might not be feasible if the platform group is properly chosen (with nilpotent subgroups of high class).
	
	 We also prove a general upper bound on Nickel's embedding and show that for any torsion-free nilpotent group the dimension of Nickel's embedding is never larger than the dimension obtained by Jennings' embedding. In order to do so, we derive a general bound on the degree (or more precisely, weight) of the multiplication polynomials, which follows by a length argument on the words introduced during the collection process.
	
	Moreover, in Section~\ref{NickJen} we consider other special classes of groups and compare Nickel's and Jennings' embedding: in free nilpotent groups both embeddings have approximately the same dimension. In contrast, for generalized Heisenberg groups, Nickel's algorithm yields an embedding of linear size whereas Jennings' embedding is quadratic.
	Also, for direct products Nickel's algorithm behaves well~-- Jennings' embedding again might lead to a large blow-up.
	Thus, although Nickel's embedding does not allow any good upper bounds either, in many situations it is much superior to Jennings' embedding in terms of the dimension of the matrix representation.
	
	Finally, in the last section we discuss some open problems related to the running time of Nickel's algorithm as well as the size of matrix representations of $\tau$-groups. 
	We start by giving some basic definitions and fix our notation.
	
	\section{Preliminaries}
	By a \emph{$\tau$-group} we mean a finitely generated torsion-free nilpotent group. $\UTZ{n}$ denotes the group of unitriangular (upper triangular and all diagonal entries equal to one) $n\times n$ matrices over the integers.

	For $a,b \in G$, the \emph{commutator} of $a$ and $b$ is defined as $[a,b] = a^{-1} b^{-1} ab$. Thus, in general, we have the relation $ab = ba [a,b]$.
	
	Let $G$ be a $\tau$-group. A \emph{Mal'cev basis} for $G$ is a tuple $(a_1, \dots, a_n)$ with $a_i \in G$ such that each $g \in G$ can be written uniquely as a normal form $g=a_1^{x_1} \cdots a_n^{x_n}$ with integers $x_1, \dots, x_n$ and such that 
	$$G = G_1 > G_2 > \cdots > G_n > 1$$
	is a central series where $G_i = \gen{a_{i}, \dots, a_n}$. Given a Mal'cev basis $(a_1, \dots, a_n)$ and constants $c_{i,j,k} \in \Z$ for $1\leq i<j<k \leq n$, a nilpotent group $G$ is uniquely defined by the relations $$[a_i, a_j]  = a_{j+1}^{c_{i,j,j+1}} \cdots a_n^{c_{i,j,n}}.$$

	The \emph{collection process} on words (sequences of letters) $w \in \oneset{a_1^{\pm1}, \dots, a_n^{\pm1}}^*$ over the generators is the  successive application of the rewriting rules
	\begin{align}
		&&&&a_j^{\eps_j} a_i^{\eps_i} &\to a_i^{\eps_i} a_j^{\eps_j} \cdot a_{j+1}^{c_{i,j,j+1}^{(\eps_i,\eps_j)}} \cdots a_{n}^{c_{i,j,n}^{(\eps_i,\eps_j)}} & \text{for $i<j$ and $\eps_i, \eps_j \in \oneset{\pm 1}$}\label{eq:collection}
	\end{align}
	where the numbers $c_{i,j,k}^{(\eps_i,\eps_j)} \in \Z$ are defined by 
	$[a_j^{\eps_j}, a_i^{\eps_i}] = c_{i,j,n}^{(\eps_i,\eps_j)}$,
	i.\,e.\ the commutators are written with respect to the Mal'cev basis.  Here, $v \to w$ for two words $v$ and $w$ means that any word of the form $u_1vu_2$ for words $u_1, u_2$ can be \emph{rewritten} in one step to $u_1wu_2$.
	An arbitrary word over the generators can be written in terms of the Mal'cev basis by using the collection process, i.\,e.\ by successively applying the rules (\ref{eq:collection}).
	
	Let $\oneset{x_1, \dots, x_n}$ be a set of \emph{variables}. A \emph{monomial} is a product of the form $ \omega = \prod_{i=1}^k x_i^{e_i}$ where $e_i\in \N$ for $i = 1, \dots, k$. Its degree is $d = \sum_{i=1}^ke_i$. If $e_i \neq 0$, we say $x_i$ \emph{appears} in the monomial $\omega$ or $\omega$ \emph{contains} $x_i$.
	A \emph{polynomial} $q = \sum_{j=1}^{m} a_j \omega_j \in \Z[x_1, \dots, x_n]$ is a sum of monomials $\omega_j$ with coefficients $a_j \in \Z$~-- we require the coefficients $a_j$ to be non-zero. Likewise, some variable $x_i$ \emph{appears} (resp.\ \emph{is contained}) in $q$ if it appears in some monomial $w_j$ with non-zero coefficient of $q$.

\subsection{Nickel's Embedding}

	Let $G$ be $\tau$-group and consider a Mal'cev basis $(a_1, \dots, a_n)$ for $G$.
	Since each $g \in G$ can be written uniquely as $g=a_1^{x_1} \cdots a_n^{x_n}$ with integers $x_1, \dots, x_n$, in particular, the product of two elements can be written in the same fashion 
	$$ a_1^{x_1}\cdots a_n^{x_n} \cdot a_1^{y_1} \cdots a_n^{y_n}= a_1^{q_1} \cdots a_n^{q_n},$$
	where the exponents $q_1,\dots, q_n$ are functions of the variables $x_1, \dots ,x_n $ and $y_1, \dots,y_n$. Hall~\cite{Hall57} showed that these functions are polynomials. We call $q_1, \dots, q_n\in \Z[x_1, \dots, x_n,y_1,\dots,y_n]$ the multiplication polynomials for the Mal'cev basis $(a_1, \dots, a_n)$. 
	Leedham-Green and Soicher \cite{LeedhamGreenS98} designed the so-called \emph{Deep Thought Algorithm} to compute the multiplication polynomials from a given Mal'cev basis $(a_1,\dots, a_n)$  and constants $c_{i,j,k} \in \Z$ for $i<j<k$ (representing the relations $[a_i, a_j]  = a_{j+1}^{c_{i,j,j+1}} \cdots a_n^{c_{i,j,n}}$). The algorithm presented in Nickel's paper takes as input a Mal'cev basis together with the corresponding multiplication polynomials (computed by the algorithm \cite{LeedhamGreenS98}) and computes a unitriangular matrix representation for $G$ over the integers.
	
	To construct the representation Nickel uses the fact that the dual space
	$$ (\Q G)^* =\{ f: \Q G \to \Q \: | \:f  \text{ is linear} \} $$
	is a G-module, where $G$ acts on $(\Q G)^*$ as follows: for $g\in G$ and $f\in (\Q G)^*$ let $f^g$ be the function defined by $h \mapsto f(h\cdot g^{-1})$ for $h\in G $. 
	
    Identifying $a_1^{x_1}\cdots a_n^{x_n}$ with $x_1,\dots ,x_n$ (this is well-defined because every group element can be uniquely written with respect to the Mal'cev basis) and writing $f(x_1,\dots, x_n)$ instead of $f(a_1^{x_1}\dots a_n^{x_n})$ allows us to view $\Q [ x_1,\dots, x_n]$ as a subset of of $(\Q G)^*$. The image of $f\in (\Q G)^*$ under the action of $g^{-1}=a_1^{y_1}\dots a_n^{y_n}$ can be described with the help of the polynomials $q_1,\dots, q_n$: for $h=a_1^{x_1}\dots a_n^{x_n}$ we have $hg^{-1}= a_1^{q_1}\dots a_n^{q_n}$. Therefore, applying $g^{-1} \in G$ to a function $f$ amounts to substituting the multiplication polynomials into $f$. If $f$ is itself a polynomial, then $f(q_1, \dots, q_n)$ is a polynomial in the variables $x_1, \dots , x_n$ and $y_1, \dots, y_n$.
	For the proof of the following lemma, we refer to~\cite{Nickel06}.
\begin{lem}[\cite{Nickel06}]\label{le:fdim}
	Let $f \in \Q[x_1,\dots,x_n]$, then the $G$-submodule of $(\Q G)^*$ generated by $f$ is finite-dimensional as a $\Q$-vector space.
\end{lem}
	The next lemma shows how to construct a finite dimensional faithful $G$-module of $(\Q G)^*$. 
	We consider the coordinate functions $t_i: G \to \Z$ for $i = 1, \dots, n$ which map $a_1^{x_1}\cdots a_n^{x_n}$ to $x_i$. Note that $t_i$ is well-defined because each element of $G$ can be written uniquely in the form $a_1^{x_1}\cdots a_n^{x_n}$. The values $t_i^g(1)$ for $i=1, \dots, n$ determine the group element $g$ uniquely. Thus, we obtain as a consequence of Lemma \ref{le:fdim}:
\begin{lem}[\cite{Nickel06}]\label{le:faithful module}
	The module $M$ generated by $t_1,\dots, t_n$ (as a submodule of $(\Q G)^*$) is a finite dimensional faithful $G$-module.
\end{lem}
	As a result of the Lemma \ref{le:faithful module}, $G$ has a matrix representation for some $n\in\N$. By choosing the order of the basis elements properly, the matrices are of unitriangular form. In the following we call this \emph{Nickel's embedding} of $G$ into $\UTZ{N}$. As part of the algorithm, Nickel also explains how the unitriangular presentation of $G$ is obtained. Next we provide a brief explanation of the algorithm, for further details, we refer to~\cite{Nickel06}.
	
\subsection{Nickel's algorithm}
	Nickel's algorithm computes the $G$-module $M$ of Lemma \ref{le:faithful module}~-- it also may be used to compute the $G$-module generated by some arbitrary list $f_1,...,f_k$ of polynomials. 
	Note that the span of a list of polynomials is a $G$-module if the image of each polynomial in the
	list under each generator of $G$ is contained in the span. Therefore, for all $j$, we consider the restricted multiplication polynomials
	$q_1^{(j)},\dots,q_n^{(j)}$ that describe the exponents of the product of a general element of the group
	and $a_j^{-1}$ : 
	$$a_1^{x_1} \cdots a_n^{x_n}·a_j^{-1}=a_1^{q_1^{(j)}} \cdots a_n^{q_n^{(j)}}$$
	
	Notice that $q_i^{(j)}=x_i$ for $i < j$ and $q_j^{j}= x_j-1$. The polynomials $q_1^{(j)}, \dots ,q_n^{(j)}$ can be obtained from the polynomials $q_1,\dots,q_n$ by setting $y_i = 0$ for $i \neq j$ and $y_j = -1$.
	This can be denoted by the shorthand notation $y_i = -\delta_{i,j}$. The image of a function $f \in {\Q G}^*$ under the action of $a_j$ is $f(q_1^{(j)}, \dots, q_n^{(j)})$.

	The algorithm starts with the coordinate functions $t_1, \dots, t_n$ (or some other list of polynomials $f_1,...,f_k$) and successively acts with $a_1$ on the basis polynomials of the previous step until the resulting polynomial lies in the span of the previous basis polynomials. Then, all basis elements are acted on with $a_2$ and so on. For a precise description as pseudocode, see Algorithm~\ref{alg:Nickel}.
	\begin{figure}[hbt]
		\begin{algorithm}{Nickel's Algorithm}\label{alg:Nickel}
			\\ \tab \coprocedure MatrixRepresentation(Mal'cev basis $a_1, \dots, a_n$,\\ multiplication polynomials $q_1, \dots, q_n$, initial polynomials $f_1, \dots, f_k$)
			\\ \tab \cobegin
			\\ \tabb Insert($B, \{f_1, \dots, f_k\}$)
			\\ \tabb \cofor $j \gets n$ \codownto $1$ \codo
			\\ \tabbb $\vec q^{(j)} = [q_1(x_1,\dots x_k, y_i=-\delta_{i,j}),\dots, q_n(x_1,\dots x_k, y_i=-\delta_{i,j})]$
			\\ \tabbb \coforeach $f\in \mathrm{Copy}(B)$ \codo
			\\ \tabbbb \corepeat
			\\ \tabbbbb $f^{a_j}\gets f(q^{(j)}_1, \dots, q^{(j)}_n)$
			\\ \tabbbbb $r \gets$ Insert($B, f^{a_j}$)
			\\ \tabbbbb $f \gets f^{a_j}$
			\\ \tabbbb \countil $r=0$ \comment{until $f^{a_j} \in \mathrm{span}(B)$}
			\\ \tabbb \coendfor
			\\ \tabb \coendfor
			\\ \tabb \coreturn $B$
			\\ \tab \coend
		\end{algorithm}
	\end{figure}
	
	The algorithm uses the \textit{Insert} routine, which adds a given polynomial to a basis of polynomials if that polynomial is not a linear combination of the basis elements. For this we fix an arbitrary ordering on the monomials. The leading monomial of a polynomial is the largest monomial with respect to that order. \textit{Insert} takes as arguments a basis of polynomials ordered with increasing leading monomials and a polynomial $f$. It subtracts from $f$ a suitable $\Q$-multiple of the polynomial in the basis with the largest leading monomial such that this monomial does not occur in the result. This operation is iterated with all basis elements in decreasing order of their leading monomials. If the final polynomial is different from the zero polynomial, it is inserted into the basis at the appropriate place and returned by the procedure.

	\section{Embedding unitriangular matrices}\label{Nickel'sbounds}
	
	Our aim of this section is to derive bounds on the dimension of the embedding which is produced by  Nickel's algorithm when embedding of unitriangular matrices into unitriangular matrices. We will both prove an exponential lower and  bound. First we will clarify our notation and collect some general facts about unitriangular matrices and the corresponding embedding.
	Let $G = \UTZ{m}$ be the unitriangular group of $m \times m$ matrices. The group $G$ has nilpotency class $c=m-1$ and Hirsch length $n= \frac{m(m-1)}{2}$.
	
	For $i<j$ let $e_{i,j}(\alpha)$ be the matrix with $ij$-th entry $\alpha$ and the rest of the entries $0$. We define $s_{i,j}(\alpha)=1+e_{i,j}(\alpha)$ and $s_{i,j}=s_{i,j}(1)$. 
	We have
	\begin{align*}
s_{i,j}^{-1}&= s_{i,j}(-1),\\
[s_{i,j}, s_{j,k}]&= \mathrlap{s_{i,k},}\hphantom{s_{i,k}(-1) } \qquad \text{and}\\ [s_{ji}, s_{k,j}]&=s_{i,k}(-1) \qquad \text{for} \: i<j<k.
	\end{align*}
	Thus, we obtain the general commutation rules
	\begin{align}
		s_{i,j}^x \;\! s_{k,\ell}^y = \begin{cases}
			s_{k,\ell}^y \;\! s_{i,j}^x & \;\text{if } i\neq \ell \text{ and } j \neq k,\\
			s_{k,\ell}^y \;\! s_{i,j}^x \;s_{i,\ell}^{xy} & \;\text{if } j = k, \\
			s_{k,\ell}^y \;\! s_{i,j}^x \;s_{k,j}^{-xy} & \;\text{if } i = \ell,
		\end{cases}\label{eq:commutators}
	\end{align}
	for arbitrary $x,y \in \Z$.
	Thus, for unitriangular matrices we have a \emph{generalized collection process} which not only operates letter by letter but which replaces factors $s_{i,j}^x \;\! s_{k,\ell}^y$ by the respective right side of (\ref{eq:commutators})~-- given that $s_{i,j}$ comes on the right of $s_{k,\ell}$ in the Mal'cev basis. Here, $x$ and $y$ even might not be only integers, but arbitrary polynomials with integer coefficients.

\begin{example}
    Let $H=\langle a_1,a_2,a_3 \mid [a_1,a_3]=[a_2,a_3] =1,[a_1,a_2]=a_3\rangle = UT_3(\Z)$ where
    
    $$a_1=\left[\begin{array}{ccc} 1&1&0 \\ 0&1&0 \\ 0&0&1  \end{array}\right], \qquad 
    a_2=\left[\begin{array}{ccc} 1&0&0 \\ 0&1&1 \\ 0&0&1  \end{array}\right], \qquad
    a_3=\left[\begin{array}{ccc} 1&0&1 \\ 0&1&0 \\ 0&0&1 \end{array}\right].$$
    
    Note that $H$ is also known as th \emph{Heisenberg group} and that $( a_1,a_2,a_3)$ is a Mal'cev basis. We will show that under Nickel's embedding $H$ is embedded into $UT_4(\Z)$. For doing that, we need to find a $\Q$-basis for the $H$-submodule of $(\Q H)^*$ generated by $\{t_1,t_2,t_3\}$ where $t_i(a_1^{x_1}a_2^{x_2}a_3^{x_3}) = x_i$.  
    In order to do so, we need to take a look at the action of powers of $\{a_1,a_2,a_3\}$ over these coordinate functions.
    We have 
    \begin{align*} t_i^{a_1^k}(a_1^{x_1}a_2^{x_2}a_3^{x_3}) & =  t_i(a_1^{x_1}a_2^{x_2}a_3^{x_3}\cdot a_1^{-k}) \\ &=   t_i(a_1^{x_1}a_2^{x_2}\cdot a_1^{-k} \cdot a_3^{x_3}) \\ & = t_i(a_1^{x_1} \cdot a_1^{-k} \cdot a_2^{x_2} \cdot a_3^{kx_2} \cdot a_3^{x_3}) \\ 
    &=t_i(a_1^{x_1-k}a_2^{x_2}a_3^{x_3+kx_2})\\ &=
    \begin{cases}
    \mathrlap{x_1-k}\hphantom{x_3+kx_2}  \;\;= t_1(a_1^{x_1}a_2^{x_2}a_3^{x_3}) - k,     &\qquad i=1,\\
    \mathrlap{x_2}\hphantom{x_3+kx_2} \;\;= t_2(a_1^{x_1}a_2^{x_2}a_3^{x_3}),          &\qquad i=2,\\
    x_3+kx_2 \;\;= t_3(a_1^{x_1}a_2^{x_2}a_3^{x_3}) + k t_2(a_1^{x_1}a_2^{x_2}a_3^{x_3}),     &\qquad i=3.\\
    \end{cases}
    \end{align*}
    Similarly,
    \begin{align*} 
    t_1^{a_2^k}(a_1^{x_1}a_2^{x_2}a_3^{x_3})&=x_1=t_1\\
    t_2^{a_2^k}(a_1^{x_1}a_2^{x_2}a_3^{x_3})&=x_2-1=t_2-k \\
    t_3^{a_2^k}(a_1^{x_1}a_2^{x_2}a_3^{x_3})&=x_3=t_3 \\
    t_1^{a_3^k}(a_1^{x_1}a_2^{x_2}a_3^{x_3})&=x_1=t_1\\
    t_2^{a_3^k}(a_1^{x_1}a_2^{x_2}a_3^{x_3})&=x_2=t_2\\
    t_3^{a_3^k}(a_1^{x_1}a_2^{x_2}a_3^{x_3})&=x_3-1=t_3-k 
    \end{align*}
    so we have to only add the constant polynomial $1$ (the constant polynomial $k$ is a multiple of it) to the set $\{t_1,t_2,t_3\}$ to obtain the $\Q$-basis $(t_1,t_2,t_3,1)$ for the $H$-submodule. We obtain the mapping
    $$a_1 \mapsto \left[\begin{array}{cccc} 1&0&0&-1\\ 0&1&1&0 \\ 0&0&1&0\\ 0&0&0&1  \end{array}\right],
    \qquad\! a_2\mapsto \left[\begin{array}{cccc} 1&0&0&0 \\ 0&1&0&0 \\ 0&0&1&-1\\ 0&0&0&1  \end{array}\right],
    \qquad\!
    a_3\mapsto \left[\begin{array}{cccc} 1&0&0&0\\ 0&1&0&-1 \\ 0&0&1&0\\ 0&0&0&1  \end{array}\right].$$
    Thus, $H$ can be embedded into $UT_4(\Z)$. 
\end{example}
	
	Now, let $(a_1,\dots, a_n)$ be an arbitrary Mal'cev basis of $\UTZ{m}$ such that $\Set{a_i}{1\leq i \leq n} = \Set{s_{i,j}}{1\leq i < j \leq m}$. That means we allow any ordering of the $s_{i,j}$~-- as long as it is still a Mal'cev basis. To keep notation simple, we write $x_{i,j} = x_k$ if $a_k = s_{i,j}$~-- likewise for $y_{i,j}$, $t_{i,j}$, and $q_{i,j}$. For the moment we use the double indices and the single indices interchangeably.
	
	Let us take a look at the multiplication polynomials $q_{i}\in \Z[x_1,\dots, x_n,y_1, \dots y_n ]$ defined by 
	$$ a_1^{x_1}\cdots a_n^{x_n} \cdot a_1^{y_1} \cdots a_n^{y_n}= a_1^{q_1} \cdots a_n^{q_n}.$$
	These multiplication polynomials are computed by applying the generalized collection process.
	Recall that the polynomials computed by Nickel's algorithm are in the span of these general multiplication polynomials after substituting the variables $y_1, \dots, y_n$ by integer values.
	
	\begin{lem}\label{lem:monomials}
		Let $1 \leq k < \ell \leq m$ and let $\omega$ be a monomial of the  multiplication polynomial $q_{k,\ell}$ (as described above). Moreover, let $d$ denote the degree of $\omega$. Then there are numbers $\lambda_\nu \in \oneset{k,\dots, \ell}$ for $0 \leq \nu \leq d$ with $\lambda_0=k$, $\lambda_d = \ell$, and $\lambda_{\nu-1} < \lambda_\nu$ for $1 \leq \nu \leq d$ such that 
		\begin{align}
			\omega = \prod_{\nu = 1}^d X_{\lambda_{\nu - 1},\lambda_\nu} \label{eq:monomial}
		\end{align}
		where $ X_{\lambda_{\nu - 1},\lambda_\nu} \in \oneset{ x_{\lambda_{\nu - 1},\lambda_\nu}, y_{\lambda_{\nu - 1},\lambda_\nu}}$.
	\end{lem}
	
	\begin{proof}
		The polynomials $q_{k,\ell}$ can be computed by the generalized collection process~-- that means an iterative application of the commutation rules of (\ref{eq:commutators}). We are going to show by induction that (\ref{eq:monomial}) holds for all polynomials appearing during this process in the exponent of any element of the Mal'cev basis. Obviously, (\ref{eq:monomial}) holds in the beginning $a_1^{x_1}\cdots a_n^{x_n} \cdot a_1^{y_1} \cdots a_n^{y_n}$ for all the exponents $x_k = x_{i,j}, y_k = y_{i,j}$ (with $d=1$).
		Exchanging two commuting elements does not change the property (\ref{eq:monomial}). 
		Exchanging $s_{i,j}^{p_{i,j}}$ and $ s_{j,k}^{p_{j,k}}$ introduces a new factor $ s_{i,k}^{p_{i,j}p_{j,k}}$ (resp.\ $ s_{i,k}^{-p_{i,j}p_{j,k}}$). By induction, we know that (\ref{eq:monomial}) holds for all monomials of $p_{i,j}$ and $ p_{j,k}$. Thus, every monomial of $p_{i,j}p_{j,k}$ is of the form
		$\left(\prod_{\nu = 1}^d X_{\lambda_{\nu - 1},\lambda_\nu}\right) \cdot \left(\prod_{\nu = 1}^{d'} X_{\lambda'_{\nu - 1},\lambda'_{\nu }}\right)$ for some $\lambda_\nu \in \oneset{i,\dots, j}$, $\lambda'_{\nu } \in \oneset{j,\dots, k}$ and $d,d' \in \N$. Since $\lambda_d= j = \lambda'_0$,  every monomial of $p_{i,j}p_{j,k}$ is of the desired form (\ref{eq:monomial}).
	\end{proof}

	\subsection{Upper bounds}	
	
	\begin{thm}\label{thm:upperboundUT}
		The dimension of the embedding of $\UTZ{m}$ produced by Nickel's algorithm is bounded by $3^{m}$ for every ordering of the Mal'cev basis $\Set{s_{i,j}}{1\leq i < j \leq m}$.
	\end{thm}
	\begin{proof}
		The $G$-module generated by the coordinate functions $\Set{t_{i,j}}{1\leq i < j \leq m}$ can be obtained as the span of all polynomials $q_{k,\ell}$ for $1 \leq k < \ell \leq m$ where the variables $y_{i,j}$ are substituted by arbitrary integer values. Obviously, this is contained in the span of all monomials of the  $q_{k,\ell}$ for $1 \leq k < \ell \leq m$ where the variables $y_{i,j}$ are substituted by integer values. If we take any monomial and substitute variables by arbitrary integer values, we obtain a multiple of when we substitute the same variables by $1$. Thus, we may assume that every variable $y_{i,j}$ is substituted by $1$. 
		
		Hence, we simply need to count the number of different monomials of the form (\ref{eq:monomial}) in Lemma~\ref{lem:monomials}. Clearly this number does not depend on the ordering of the Mal'cev basis. 
		Let $N_{k,\ell}$ be the number of monomials of the form (\ref{eq:monomial}) with $\lambda_0=k$ and $\lambda_d= \ell$. Then we have 
		$ \sum_{1\leq k < \ell \leq m} N_{k,\ell}$
		as an upper bound of the dimension of the embedding.
		
		Now, each pick of indices $\lambda_\nu \in \oneset{k,\dots, \ell}$ for $0 \leq \nu \leq d$ with $\lambda_0=k$, $\lambda_{\nu - 1} < \lambda_{\nu }$ and $\lambda_d= \ell$ corresponds to a $d-1$-element subset of $\oneset{k+1,\dots, \ell-1}$. 
		Moreover, for each pick of indices there are precisely $2^d$ ways to assign variables $x_{i,j}$ or $y_{i,j}$ to the $X_{i,j}$s.
		As $d$ ranges from $1$ to $\ell - k$, we obtain \begin{align*}
		N_{k,\ell} = \sum_{d = 1}^{\ell - k}\binom{\ell - 1 - k}{d-1}2^d =  2\sum_{d = 0}^{\ell -1 - k}\binom{\ell - 1 - k}{d}2^d = 2 (1 + 2)^{\ell - 1 - k}.
		\end{align*}
		Thus, it follows
		\begin{align*}
			\sum_{1\leq k < \ell \leq m} N_{k,\ell}  = 2\!\!\!\sum_{1\leq k < \ell \leq m} 3^{\ell - k - 1}
			 = 2\!\sum_{1 < \ell \leq m}\, \sum_{1\leq i < \ell} 3^{i - 1}
		 \leq \sum_{1 < \ell \leq m} 3^{\ell - 1}
			 \leq 3^{m}.%
		\end{align*}%
	\end{proof}

	\subsection{Lower bounds}
	For proving lower bounds, we use the Mal'cev basis for $G= ( a_1, a_2, \dots,  a_n )$ (which is the standard Mal'cev basis, see e.\,g.\ \cite{Nickel06}) with
	\begin{align*}
		&&	&&a_k &= s_{i,j} &\text{for }&k = i + \sum_{\ell = 1}^{j-i-1} (m-\ell).&&
	\end{align*}
	Thus, the order of the basis elements can be depicted as follows:
	$$
	\begin{pmatrix}
	1  &a_1& a_m &  & \cdots & a_{n} \\
	& 1 & a_2  &  a_{m+1}& &\\
	&   &  1 & a_3&  \ddots& \vdots \\
	&   &       &  1    & \ddots & a_{2m-3} \\
	& 0 &     &  & \ddots     & \hspace{-2mm}a_{m-1} \\
	&   &      & &        & 1
	\end{pmatrix}
	$$ 
In particular, we have
	$a_i = s_{i(i+1)}$ for $1 \leq i\leq m-1$.	
	From now on we will denote 
	\begin{align*}
		a_1^{x_1} a_2^{x_2} \cdots a_{m-1}^{x_{m-1}} a_m^{x_m} &\cdots a_{2m-2}^{x_{2m-2}} \cdots a_{n-2}^{x_{n-2}} a_{n-1}^{x_{n-1}} a_n^{x_n} 
		\\&= s_{1,2}^{x_{1,2}} s_{2,3}^{x_{23}} \cdots s_{m-1,m}^{x_{m-1,m}} s_{1,3}^{x_{1,3}} \cdots s_{1,m-1}^{x_{1,m-1}} s_{2,m}^{x_{2,m}} s_{1,m}^{x_{1,m}}.
	\end{align*} 
	That means, as in the previous section, instead of variables with a single index $x_k$ we write the variables with two indices $x_{i,j}$.
	
Following the idea used in the embedding, we are going to have a look at the action of $a_i^{k_i}$ on the last coordinate function $t_n$ for $1 \leq i \leq \floor{\frac{m}{2}} - 1$ and provide a lower bound for the number of linearly independent polynomial exponents. That means we apply the generalized collection process to
	\begin{align}
		a_1^{x_1}\cdots a_n^{x_n}\cdot a_i^{-k}.\label{eq:rightmult}
	\end{align} 
	This can be done in two steps: first, $ a_i^{-k}$ is moved to the left, and the respective commutators are introduced. In the second step, the newly introduced commutators are moved to the right until they remain at their correct place. It is a crucial property of a Mal'cev basis that these newly introduced commutators only ``travel'' to the right.
	For our purposes $k=1$ in (\ref{eq:rightmult}) is sufficient.
	To prove our result, Theorem \ref{thm:complexity}, we will first look at the case when $i=1$, which will be the starting point to see that there are exponentially many basis elements. The reason why we look at the case $i=1$, is simply because the polynomial exponent with the highest degree will appear in this case since $a_1^{-1}$ is the element that has to ``travel'' over all the elements $a_j^{x_j}$ for $2 \leq j \leq n$ in (\ref{eq:rightmult}). So we start with
	\begin{align*}
		a_1^{x_1} a_2^{x_2}\cdots {}& a_{m-1}^{x_{m-1}} a_m^{x_m} \cdots a_{n-1}^{x_{n-1}} a_n^{x_n} \cdot a_1^{-1} \\&=
		s_{1,2}^{x_{1,2}}s_{2,3}^{x_{23}}\cdots s_{m-1,m}^{x_{m-1,m}} s_{1,3}^{x_{1,3}} \cdots s_{1,m-1}^{x_{1,m-1}} s_{2,m}^{x_{2,m}} s_{1,m}^{x_{1,m}} \cdot s_{1,2}^{-1}.
	\end{align*} 
	Note that $s_{1,2}$ commutes with all other elements except $s_{2,3},s_{2,4}, \dots , s_{2,m}$, and  we know that by (\ref{eq:commutators})
	$$ s_{2,j}^{x_{2,m}} s_{1,2}^{-1} = s_{1,2}^{-1} s_{2,j}^{x_{2,m}}\, s_{1,j}^{x_{2,m}}$$
	for $3 \leq j\leq m$. 
	So, as $s_{1,2}$ moves to the left in the product, the first element that it has to pass over and that it does not commute with is $s_{2,m}$~-- we obtain
	\begin{align*}
		s_{1,2}^{x_{1,2}} s_{2,3}^{x_{2,3}} \cdots {}& s_{m-1,m}^{x_{m-1,m}} s_{1,3}^{x_{1,3}} \cdots s_{1,m-1}^{x_{1,m-1}} s_{2,m}^{x_{2,m}} s_{1,m}^{x_{1,m}}\cdot s_{1,2}^{-1} 
		\\ &= s_{1,2}^{x_{1,2}} s_{2,3}^{x_{2,3}}\cdots s_{m-1,m}^{x_{m-1,m}} s_{1,3}^{x_{1,3}} \cdots s_{1,m-1}^{x_{1,m-1}} \cdot s_{1,2}^{-1} \cdot s_{2,m}^{x_{2,m}} \, s_{1,m}^{x_{1,m} + x_{2,m}}.
	\end{align*} 
Next it does not commute with $s_{2,m-1}^{x_{2,m-1}}$~-- however, the commutator is $s_{1,m-1}^{x_{2,m-1}}$; so, the exponent of $s_{1,m}$ does not change. Continuing this way moving up to $s_{2,\ell}$ for $\ell=\floor{\frac{m}{2}}+1$, we have, for example, when $m$ is even; 
	\begin{align*}
		s_{1,2}^{x_{1,2}} s_{2,3}^{x_{2,3}} \cdots {}& s_{m-1,m}^{x_{m-1,m}} s_{1,3}^{x_{1,3}} \cdots s_{1,m-1}^{x_{1,m-1}} s_{2,m}^{x_{2,m}} s_{1,m}^{x_{1,m}}\cdot s_{1,2}^{-1} \\
		&=s_{1,2}^{x_{1,2}} s_{2,3}^{x_{2,3}} \cdots  s_{m-1,m}^{x_{m-1,m}} s_{1,3}^{x_{1,3}} \cdots  s_{1,\ell-1}^{x_{1,\ell-1}} \cdot s_{1,2}^{-1} \cdot s_{2,\ell}^{x_{2,\ell}} \cdots s_{\ell,m}^{x_{\ell,m}}  s_{1,\ell}^{x_{1,\ell} + x_{2,\ell}} \\
		& \qquad\cdots s_{1,m-1}^{x_{1,m-1} + x_{2,m-1}} s_{2,m}^{x_{2,m}} s_{1,m}^{x_{1,m} + x_{2,m}+x_{2,\ell} x_{\ell,m}}.
	\end{align*}
	That means for the first time a non-linear polynomial appears.
	As $s_{1,2}^{-1}$ moves further to the left, the exponents of $s_{1,j}$'s for $ 3 \leq j < m$ become quite complicated to express; therefore, we are going to only focus on what is being added to the exponent of $s_{1,m}$ once $s_{1,2}^{-1}$ passes all $s_{2,j}$~-- we use $\displaystyle\star$ as a placeholder for any polynomial:
	\newcommand{\sOneTwo}{{\color{red} s_{1,2}^{-1} }}
	\begin{align*}
		s_{1,2}^{x_{1,2}} &\cdots s_{m-1,m}^{x_{m-1,m}} s_{1,3}^{x_{1,3}} \cdots s_{1,m-1}^{x_{1,m-1}} s_{2,m}^{x_{2,m}} s_{1,m}^{x_{1,m}}\cdot \sOneTwo \\
		&= s_{1,2}^{x_{1,2}} s_{2,3}^{x_{2,3}} \cdot \sOneTwo \cdot s_{3,4}^{x_{3,4}} \cdots s_{m-1,m}^{x_{m-1,m}}\, s_{1,3}^{\displaystyle\star} \cdots s_{1,m-1}^{\displaystyle\star} s_{2,m}^{x_{2,m}} \, s_{1,m}^{x_{1,m} +x_{2,m} + P'_{1,m} } \\
		&= s_{1,2}^{x_{1,2}} \cdot \sOneTwo \cdot s_{2,3}^{x_{2,3}} \cdot {\color{blue} s_{1,3}^{x_{2,3}}} \cdot s_{3,4}^{x_{3,4}} \cdots s_{m-1,m}^{x_{m-1,m}}\, s_{1,3}^{\displaystyle\star} \cdots s_{1,m-1}^{\displaystyle\star} s_{2,m}^{x_{2,m}}\, s_{1,m}^{x_{1,m} +x_{2,m} + P'_{1,m} } \\
		&= s_{1,2}^{x_{1,2}{\color{red}-1}} s_{2,3}^{x_{2,3}} s_{3,4}^{x_{3,4}} \cdot {\color{blue} s_{1,3}^{x_{2,3}}} \cdot {\color{blue} s_{1,4}^{x_{2,3}x_{3,4}}} s_{4,5}^{x_{4,5}} \cdots s_{m-1,m}^{x_{m-1,m}}\, s_{1,3}^{\displaystyle\star} \cdots s_{1,m-1}^{\displaystyle\star} s_{2,m}^{x_{2,m}} \, s_{1,m}^{x_{1,m} +x_{2,m} + P'_{1,m} } \\
		&= s_{1,2}^{x_{1,2}{\color{red}-1}} s_{2,3}^{x_{2,3}} s_{3,4}^{x_{3,4}} s_{4,5}^{x_{4,5}} \cdot {\color{blue} s_{1,5}^{x_{2,3}x_{3,4} x_{4,5}}} \cdot s_{5,6}^{x_{5,6}} \cdots s_{m-1,m}^{x_{m-1,m}}\, s_{1,3}^{\displaystyle\star} \cdots s_{1,m-1}^{\displaystyle\star} s_{2,m}^{x_{2,m}}\, s_{1,m}^{x_{1,m} +x_{2,m} + P''_{1,m} } \\
		&= \cdots \\
		&= s_{1,2}^{x_{1,2}{\color{red}-1}} s_{2,3}^{x_{2,3}} \cdots s_{m-1,m}^{x_{m-1,m}} \cdot {\color{blue} s_{1,m}^{x_{2,3}\cdots x_{m-1,m}}} \cdot s_{1,3}^{\displaystyle\star} \cdots s_{1,m-1}^{\displaystyle\star} s_{2,m}^{x_{2,m}} \, s_{1,m}^{x_{1,m} +x_{2,m} + P'''_{1,m} } \\
		&= s_{1,2}^{x_{1,2}{\color{red}-1}} s_{2,3}^{x_{2,3}} \cdots s_{m-1,m}^{x_{m-1,m}} \cdot s_{1,3}^{\displaystyle\star} \cdots s_{1,m-1}^{\displaystyle\star} s_{2,m}^{x_{2,m}} \, s_{1,m}^{x_{1,m} + x_{2,m} + {\color{blue} x_{2,3} \cdots x_{m-1,m}} + P_{1,m} }. 
	\end{align*}
	%
	\newcommand{\q}{\ov q}
	\newcommand{\qm}{q}
	\newcommand{\xI}{x_{i,i+1}}
	Finally we have the polynomial exponent of $s_{1,m}$, which will be denoted by $\qm$, as
	$$\qm=x_{1,m} + x_{2,m} +\prod_{i=2}^{m-1} \xI + P_{1,m}.$$
	
	Note that $\qm = t_{1,m}^{s_{1,2}}$ where $t_{1,m} = t_n$ is the $n$-th coordinate function. Also notice that, as a consequence of Lemma~\ref{lem:monomials} (with $y_1=1, y_i=0$ for $i>1$), $P_{1,m}$ does not contain any linear monomials (and it is non-zero for $m$ large enough). 
	From now on, let $\mu = \floor{\frac{m}{2}}$.

	\begin{lem}\label{lem:abeliansubgroup}
		The subgroup generated by $\Set{s_{i,j}}{j - i \geq m - \mu} $ is abelian. 
	\end{lem}
	
	\begin{proof}
		Assume $s_{i,j}, s_{k,\ell} \in \Set{s_{i,j}}{j - i \geq m- \mu}$ do not commute. By (\ref{eq:commutators}), it suffices to look at the case $j =k$. Then we have $$\ell \geq m- \mu + j \geq m- \mu + m- \mu + i \geq 2m- 2 \mu + 1 > m.$$%
	\end{proof}
	
	Note that the basis elements $\Set{s_{i,j}}{j - i \geq m - \mu}$ form a suffix of the Mal'cev basis. This observation leads to the following:
	\begin{lem}\label{lem:linearmonomials}
		Among all polynomials obtained from the coordinate functions by the action of $G$, the variables $x_{i,j}$ for $j - i \geq m - \mu$ appear only in linear monomials.
	\end{lem}
	\begin{proof}
	It follows from Lemma~\ref{lem:monomials} that variables $x_{i,j}$ for $j - i \geq m - \mu$ can only occur as exponents of some $ s_{k,\ell} \in \Set{s_{i,j}}{j - i \geq m- \mu}$.
	When moving some $s_{i,j}^{-1}$ to the left, only linear monomials are introduced. When commutators of $s_{i,j}^{-1}$ and something else move to the right, they either belong to the subgroup generated by $\Set{s_{i,j}}{j - i \geq m- \mu}$~-- and, thus, by the previous lemma commute with everything in that subgroup~--, or they to not belong to that subgroup and, therefore, remain on the left of the elements $\Set{s_{i,j}}{j - i \geq m- \mu}$ (since $\Set{s_{i,j}}{j - i \geq m - \mu}$ form a suffix of the Mal'cev basis).
	\end{proof}

	
	Next, we look at the action on the multiplication polynomials $q_{i,j}$.
	\begin{lem}\label{lem:longfactor}
		Let $\lambda \in \oneset{1, \dots, m-1}$ and $(i,j) \neq (m-1,m)$ and let $g$ be in the subgroup generated by $\Set{s_{i,i+1}}{1\leq i \leq \lambda}$. Then any monomial of $q_{i,j}^{g}$ which contains the variable $x_{m-1,m}$ also contains a factor $\prod_{i = \lambda+1}^{m-1} x_{i,i+1}$.
	\end{lem}
	\begin{proof}
		For $\lambda = m-2$, the statement is trivial. Hence, let $\lambda + 1 < m - 1$.
		If the variable $x_{m-1,m}$ appears in some  $q_{i,j}^{g}$ with $(i,j) \neq (m-1,m)$, it must have been first introduced as a commutator of some $s_{k,\ell}^p$ and $s_{m-1,m}^{x_{m-1,m}}$ for some polynomial $p$. In particular, $\ell = m-1$. We show that every monomial of $p$ contains the variable $x_{m-2,m-1}$; the lemma then follows by induction.
		
		Indeed, assume that there is a monomial $\omega$ in $p$ which does not contain a variable $x_{m-2,m-1}$. Then, in particular, $s_{k,\ell}^{\omega}$ cannot be written as a commutator of $s_{m-2,m-1}^{x_{m-2,m-1}}$. However, by the ordering of the Mal'cev basis and the choice of $g$, this means that $s_{k,\ell}^{\omega}$ is contained in the subgroup generated by $\Set{s_{i,i+1}}{i \leq m-3}$, which commutes with $s_{m-1,m}$~-- a contradiction.
	\end{proof}

	Now, let us return to the polynomial $\qm = t_{1,m}^{s_{1,2}}$. As we computed above, we have $\qm = x_{1,m} + x_{2,m} + \prod_{i=2}^{m-1} \xI +P_{1,m}$ for some polynomial $P_{1,m}$ which does not contain the monomial $\prod_{i=2}^{m-1} \xI$ and linear terms. We will denote $$\q=\prod_{i=2}^{m-1} \xI + P_{1,m} \equiv \qm \mod \text{ linear terms}.$$ Clearly $\q$ is contained in 
	the $G$-module generated by $\Set{t_{i,j}}{1\leq i < j \leq m}$. Later we will consider the action of a subgroups of $G$ on $\q$, but first we need to summarize some more facts on $\q$.

	\begin{lem}\ \label{lem:endvariables2}
		\begin{itemize}
			\item Every monomial of $\q$ contains a variable $x_{i,m}$ with $i > \mu$. 
			\item In $\q$ the variable $x_{m-1,m}$ only appears in the monomial $\prod_{i=2}^{m-1} \xI$.
		\end{itemize}
	\end{lem}
	
	\begin{proof}
		By Lemma \ref{lem:monomials}, every monomial of $\q$ must contain a variable $x_{i,m}$ for some $i$. As $\q$ does not contain linear monimals, we know that $i > \mu$ by Lemma \ref{lem:linearmonomials}.  The  second statement is a direct consequence of Lemma \ref{lem:longfactor}.
	\end{proof}

	\newcommand{\subs}{S}
	\newcommand{\wS}{w_\subs}
	\newcommand{\tS}[2]{t_{#1,#2}^{w_S}}
	
	For some subset $\subs \subseteq \oneset{2, \dots , \mu}$ (recall that $\mu = \floor{\frac{m}{2}}$), we define $X_\subs = \prod_{i\in \subs} x_{i,i+1}$ and $\wS = \prod_{i\in \subs} s_{i,i+1}$ where the $s_{i,i+1}$ are ordered ascending by their index $i$ in the product.

	We are going to show that the polynomials $\Set{\q^{\wS}}{\subs \subseteq \oneset{2, \dots, \mu}}$ are all linearly independent and, thus, establish an exponential lower bound for Nickel's embedding. Computing the action of $\wS$ on some polynomial can be done by computing the action on the coordinate functions $t_{1,2}, \dots, t_{1,m}$ and then substituting the variables $x_{i,j}$ of $\q$ by the respective $t_{i,j}^{\wS}$.

	Therefore, let us consider the polynomials $t_{i,j}^{w_S}$. Recall that we have
	\begin{align*}
		s_{1,2}^{x_{1,2m}} \cdots s_{1,m}^{x_{1,m}} \cdot w_S^{-1} = s_{1,2}^{\tS{1}{2}} \cdots s_{1,m}^{\tS{1}{m}}.
	\end{align*}

	\begin{lem}\label{lem:Tinvariant}
		$\tS{i}{m} = t_{i,m}$ for $\mu < i \leq m$.
	\end{lem}
	\begin{proof}
		Let us take a look at the multiplication polynomial $q_{i,m}(x_{1,2}, \dots, x_{1,m}, y_{1,2}, \dots, y_{1,m})$. Every monomial in $\tS{i}{m}$ is obtained from a monomial in $q_{i,m}$ by substituting the variables $y_{j,k}$ by integers. By the choice of $ \wS$, only variables $y_{j,j+1}$ with $ j \leq \mu $ are substituted by non-zero values. Thus, a newly introduced monomial must have been obtained from a monomial containing a variable $y_{j,j+1}$ for some $ j \leq \mu$. 			
		By Lemma~\ref{lem:monomials} that means new monomials can only be introduced to $q_{i,m}$ for $ i \leq \mu$.
	\end{proof}

	Finally, we need the following fact, which is true because no $s_{i,i+1}$ is a commutator.
	\begin{align}
		\tS{i}{i+1} = \begin{cases}
			x_{i,i+1} - 1 & \text{if } i \in S\\
			x_{i,i+1} & \text{otherwise }\\
		\end{cases} \text{ for }1 \leq i \leq m-1.\label{eq:WeightOne}
	\end{align}

 	Now, we are ready to prove the exponential lower bound on the dimension of Nickel's embedding.
	\begin{thm}\label{thm:complexity} 
	Assume the group $G=\UTZ{m}$ is given with with respect to the "standard" Mal'cev basis described above.
	Then for the embedding  computed by Nickel's algorithm of $G$ into  $\UTZ{N}$, we have $N \geq 2^{\floor{\frac{m}{2}}-1}$.
	\end{thm}
	\begin{proof}
		Computing $\q^{\wS}$ means we substitute all variables $x_{i,j}$ in $\q$ by the respective polynomials $\tS{i}{j}$. Let $M=\oneset{2,\dots,m-1}$.
		By (\ref{eq:WeightOne}), the monomial $\prod_{i=2}^{m-1}\xI$ becomes 
		$$p_S = \prod_{i \in  S} (\xI - 1)\:\cdot\!\! \prod_{i\in M\setminus S} \xI =\sum_{T \subseteq S} (-1)^{\vert T\vert} \prod_{i\in M\setminus T} \xI.$$
		As every of the $p_S$ has a unique monomial of highest degree, the set $\Set{p_S}{S \subseteq \oneset{2, \dots, \mu}}$ is linearly independent.
		
		We want to show that all these polynomials $p_S$ also appear as summands of $\q^{\wS}$. We can guarantee that by showing that no monomial of the form $\prod_{i\in M\setminus T}\xI$ appears as substitution of any of the other monomials of $\q$~-- then no cancellation can occur.
		By Lemma \ref{lem:endvariables2}, every monomial of $\q - \prod_{i=2}^{m-1}\xI$ contains a variable $x_{i,m}$ for some $\mu < i < m$. Moreover, by Lemma~\ref{lem:Tinvariant} every variable $x_{i,m}$ for  $\mu < i < m$ is substituted by itself. Thus, no monomial appearing in a substitution of $\q - \prod_{i=2}^{m-1}\xI$ is equal to a monomial of the form $\prod_{i\in M\setminus T}\xI$.

		Thus, none of the monomials of the form $\prod_{i\in M\setminus T}\xI$ for any $T \subseteq \oneset{2, \dots, \mu}$ gets cancelled in the substitution of $\q$. Therefore, also every polynomial $\q^{\wS}$ has its unique monomial of highest degree in the variables $x_{2,3}, \dots, x_{m-1,m}$, and hence, also the set $\Set{\q^{\wS}}{S \subseteq \oneset{2, \dots, \mu}}$ is linearly independent. Thus, we have completed the proof of Theorem \ref{thm:complexity}.
	\end{proof}

	\subsection{Reordering the Mal'cev basis}
	
	If we reorder the Mal'cev basis such that $s_{i,j}$ comes on the left of $s_{k,\ell}$ if and only if $j < \ell$ or $j=\ell$ and $i > k$, that is according to the scheme
	$$
	\begin{pmatrix}
	1  &a_1& a_3   & a_6 & \cdots & a_{n} \\
	& 1 & a_2   & a_{5}& & a_{n-1}\\
	&   &   1   & a_4 & &\\
	&   &       & 1      &  &\smash{\vdots}  \\
	& 0 &       &    &  \ddots & \\ 
	&   &       &      &  & 1
	\end{pmatrix},
	$$ 
	then Nickel's algorithm produces an embedding of dimension $n + 1$. 
	In order to see this, proceed as follows: Consider the multiplication polynomials $q_i^{(j)}$ defined by
	$$a_1^{x_1} \cdots a_n^{x_n}·a_j^{-1}=a_1^{q_1^{(j)}} \cdots a_n^{q_n^{(j)}}.$$
	Now, for $a_j = s_{k,\ell}$ we know that $a_j$ commutes with all elements which are on the right of $a_j$ in the Mal'cev basis~-- except elements of the $\ell$-th row. To compute the multiplication polynomials $q_i^{(j)}$, we can move $a_j^{-1}$ step by step to the left introducing the respective commutators. As a commutator of $a_j = s_{k,\ell}$ and some $s_{\ell,\lambda}^{x_{\ell,\lambda}}$, the element $s_{k,\lambda}^{\pm x_{\ell,\lambda}}$ is introduced. Since $s_{k,\lambda}$ belongs to the same column as $s_{\ell,\lambda}$ (and in the Mal'cev basis there are only elements of the same column between the two elements), it can be moved to its correct position in the Mal'cev basis without introducing further commutators. Therefore, only linear and constant monomials occur in the multiplication polynomials. Thus, $(t_n, \dots, t_1, 1)$ is a basis of the $G$-module generated by the coordinate functions $t_1, \dots, t_n$ and all associated matrices are of triangular form.

	\section{Jennings' embedding}\label{sec:Jennings}
	
	For the proof of some general upper bounds for Nickel's embedding, we need some basic facts about Jennings' embedding \cite{Jennings55}.
	Therefore, let us briefly describe how that embedding works.
	Let $G$ be a $\tau$-group with Mal'cev basis $(a_1, \dots, a_n)$ and nilpotency class $c$. 
	We denote $\Gamma_{i + 1} = [ \Gamma_i, G]$ with $\Gamma_1 =G$ and $\tau_i = \Set{g\in G}{\exists\, k\; : \; g^k \in \Gamma_i}$ the \emph{isolator} of $\Gamma_i$. The \emph{weight} $\nu(g)$ of some $g\in G$ is defined as the largest $i$ with $g \in \tau_i$.
	
	The embedding is given by the right action of $G$ on $\Q G/I^{c+1}$ where $\Q G$ is the group ring with rational coefficients and $I = \Set{\sum_{g\in G} \lambda_g g}{\sum_{g\in G} \lambda_g = 0}$ is the augmentation ideal.
	
	We describe a basis of Jennings' embedding according to \cite[Lem.\ 7.3]{Hall57}.
	Set $u_i= 1-a_i \in \Q G$ and let $M\in \N$ be arbitrary. Then the set of all products $v= v_1 \dots v_n$ in which each $v_i$ has one of the following forms,
	\begin{itemize}
		\item $u_{i}^{r_i}\hphantom{u_{i}^Ma_{i}^{-s_i}} \quad \text{with } r_i\in \N$
		\item $u_{i}^Ma_{i}^{-s_i}\hphantom{u_{i}^{r_i}} \quad \text{with }  s_i \in \N \setminus \oneset{0}$, 
	\end{itemize}
	forms a basis of $\Q G$.
	For a basis element $v=v_1v_2\dots v_n$ of $\Q G$, the \emph{weight} is defined as
	\begin{itemize}
		\item $\mathrlap{\nu(u_i)=\nu(a_i)} \hphantom{\nu(v)=\sum_{i}^{n} r_i\nu(u_i)\hphantom{M}}$ (that means $\nu(u_i) = k $ if and only if $a_i \in \tau_k$, but $a_i \not\in \tau_{k+1}$),
		\item $ \displaystyle \nu(v)=\sum_{i}^{n} r_i\nu(u_i)\hphantom{M} $ if all $v_i$ have the form $u_i^{r_i}$,
		\item $\nu (v)= M\hphantom{\sum_{i}^{n} r_i\nu(u_i)}$ if at least one $v_i=u_i^Ma_i^{-s_i}$. 
	\end{itemize}
	Then $\Set{v}{\nu(v) \geq M}$ is a basis of $I^M$.
	That means on the other hand that $\Set{v}{\nu(v) \leq c}$ is a basis of the faithful $G$ module $\Q[G]/I^{c+1}$ where the action of $G$ is the right multiplication and basis elements of weight greater than $c$ are simply ignored.
	The matrices defined by the action have only integral entries and all diagonal entries are one.
	Moreover, if the basis elements are ordered according to their weight, then they are of triangular shape.

	\section{General upper bounds for Nickel's algorithm}

	Now, let $G$ be an arbitrary $\tau$-group with Mal'cev basis $(a_1, \dots, a_n)$.
	Let us again consider the multiplication polynomials $q_1, \dots, q_n \in \Z[x_1, \dots, x_n, y_1, \dots, y_n]$ defined by
	\begin{align}
		a_1^{x_1}\cdots a_n^{x_n} \cdot a_1^{y_1} \cdots a_n^{y_n} = a_1^{q_1} \cdots a_n^{q_n}.\label{multiplicationpolynomials}
	\end{align}
	For a monomial $\omega = x_1^{e_1} \cdots x_n^{e_n} \cdot y_1^{f_1} \cdots y_n^{f_n}$, we define its \emph{weight} $\nu(\omega)$ to be $\sum (e_i + f_i)\nu(a_i)$ (where $\nu(a_i)$ is the weight of $a_i$). The weight of a polynomial $q$ is the maximal weight of its monomials $\nu(\sum_j \omega_j) = \max_j\nu(\omega_j)$~-- thus, if all variable have weight one, the weight agrees with the degree of a polynomial. 
	
    By \cite[7.4]{LeedhamGreenS98}, the multiplication polynomials have degree bounded by the nilpotency class $c$. However, we need a better bound. It is straightforward to see that~-- at least in the case of unitriangular matrices~-- the following bound is tight.
	
\begin{lem}\label{lem:weightbound}
	For all $i$ we have $\nu(q_i) \leq \nu(a_i)$.
\end{lem}
	
	Lemma~\ref{lem:weightbound} can be derived from Osin's theorem on subgroup distortion \cite[Thm.\ 2.2]{Osin01}. Since a direct proof is not much longer, we give the full prove here. Actually, parts of this proof follow the same ideas as in \cite{Osin01}.

\begin{proof}
	Let us consider the collection process on words $w \in \oneset{a_1^{\pm1}, \dots, a_n^{\pm1}}^*$ over the generators: that means a successive application of the rewriting rules
	\begin{align*}
		&&&&a_j^{\eps_j} a_i^{\eps_i} &\to a_i^{\eps_i} a_j^{\eps_j} \cdot a_{j+1}^{c_{i,j,j+1}^{(\eps_i,\eps_j)}} \cdots a_{n}^{c_{i,j,n}^{(\eps_i,\eps_j)}} & \text{for $i<j$ and $\eps_i, \eps_j \in \oneset{\pm 1}$}
	\end{align*}
	where the numbers $c_{i,j,k}^{(\eps_i,\eps_j)} \in \Z$ are defined by 
	$$[a_j^{\eps_j}, a_i^{\eps_i}] = c_{i,j,n}^{(\eps_i,\eps_j)},$$
	i.\,e.\ the commutators are written with respect of the Mal'cev basis. 
	Here $a_k^{e}$ for $e \in \Z$ stands for the word $a_k \cdots a_k$ with $e$ factors $a_k$ if $e$ is positive and for the word $a_k^{-1} \cdots a_k^{-1}$ with $\abs{e}$ factors $a_k^{-1}$ if $e$ is negative. Note that we do not define any cancellation rules (in particular $a_k^{-1}a_k$ cannot be replaced by the empty word)~-- hence, if a letter appears at some time during the collection process, it will remain there throughout the whole collection process (only the position may change due to the rewriting steps).
	
	We denote the number of occurrences of letters $a_i^{\pm1}$ in $w$ with $\abs{w}_i$.
	Now, consider a word $\hat w$ obtained from $w$ by any number of rewriting steps and count the appearances of letters $a_k^{\pm1}$ for some $k$.
	Any letter  $a_k^{\pm1}$ in $\hat w$ either was there already in $w$ or it was introduced in the collection process. In the latter case, in particular, we had $c_{i,j,k}^{(\eps_i,\eps_j)} \neq 0$ for some $i<j<k$ and  $\eps_i, \eps_j \in \oneset{\pm 1}$. The number of letters  $a_k^{\pm1}$  which were introduced when exchanging $a_j^{\eps_j}$ and $ a_i^{\eps_i}$ for fixed $i<j<k$ and  $\eps_i, \eps_j \in \oneset{\pm 1}$ is bounded by  $\abs{c_{i,j,k}^{(\eps_i,\eps_j)}}$ times the product of the total number of occurrences of $a_j^{\eps_j}$ and $ a_i^{\eps_i}$ in $\hat w$ (this is because letters never disappear).
	Thus, we have
	\begin{align}
		\abs{\hat w}_k \leq \abs{w}_k +  \sum_{c_{i,j,k}^{\!(\eps_i,\eps_j)} \!\neq 0} \abs{c_{i,j,k}^{(\eps_i,\eps_j)}} \abs{\hat w}_i \abs{\hat w}_j\label{eq:wkbound}
	\end{align}
	Next we are going to show the following:
	\begin{align}
		\abs{\hat w}_k &\leq 
		C \cdot  \!\!\!\!\!\!\!\!\! \sum_{\stackrel{(e_1, \dots, e_n) \in \N^n}{\sum_i e_i\nu(a_{i})\leq \nu(a_k)}}  \prod_i\abs{w}_{i}^{e_i}\label{eq:lengthbound}
	\end{align}
	for some constant $C$ (depending on the $c_{i,j,k}^{(\eps_i,\eps_j)}$ and $n$ and $c$). Note that (\ref{eq:lengthbound}) is a slight variation of the well-known fact \cite[Thm.\ 2.3]{MacDonaldMNV15}. In order to prove (\ref{eq:lengthbound}), we procede by induction.
	For $\nu(a_k)=1$, (\ref{eq:lengthbound}) holds trivially. Now, let $\nu(a_k)>1$. Note that $c_{i,j,k}^{(\eps_i,\eps_j)} \neq 0 $ implies that $\nu(a_k) \geq \nu(a_i) + \nu(a_j)$ (this is a general fact~-- see e.\,g.\ \cite[Thm.\ 5.3 (4)]{MKS66}). Therefore, we have by (\ref{eq:wkbound})
	\begin{align*}
		\abs{\hat w}_k &\leq \abs{w}_k +  \sum_{c_{i,j,k}^{(\eps_i,\eps_j)} \neq 0} \abs{c_{i,j,k}^{(\eps_i,\eps_j)}} \abs{\hat w}_i \abs{\hat w}_j\\
		& \leq \abs{w}_k +   \!\!\!\!\sum_{\stackrel{(i,j)}{\nu(a_i) + \nu(a_j) \leq \nu(a_k)}}  \!\!\!\!\abs{c_{i,j,k}^{(\eps_i,\eps_j)}} \left( C \cdot  \!\!\!\!\!\!\!\!\! \sum_{\stackrel{(e_1, \dots, e_n) \in \N^n}{\sum_\ell e_\ell\nu(a_{\ell})\leq \nu(a_i)}}  \prod_\ell\abs{w}_{\ell}^{e_\ell} \right) 
		\left( C \cdot  \!\!\!\!\!\!\!\!\! \sum_{\stackrel{(e_1, \dots, e_n) \in \N^n}{\sum_\ell e_\ell\nu(a_{\ell})\leq \nu(a_j)}}  \prod_\ell\abs{w}_{\ell}^{e_\ell}\right)\\
		&\leq C' \cdot  \!\!\!\!\!\!\!\!\! \sum_{\stackrel{(e_1, \dots, e_n) \in \N^n}{\sum_i e_i\nu(a_{i})\leq \nu(a_k)}}  \prod_i\abs{w}_{i}^{e_i}
	\end{align*}
	for some properly chosen $C'$ (recall that the sums range over a constant number of indices). Thus, we have shown (\ref{eq:lengthbound}).
	
	Then starting the collection process with the word $w = a_1^{x_1}\cdots a_n^{x_n} \cdot a_1^{y_1} \cdots a_n^{y_n}$, we obtain 
	\begin{align}
		\abs{q_k(x_1, \dots, x_n, y_1, \dots, y_n)} &\leq 
		C \cdot  \!\!\!\!\!\!\!\!\! \sum_{\stackrel{(e_1, \dots, e_n) \in \N^n}{\sum_i e_i\nu(a_{i})\leq \nu(a_k)}}  \prod_i(\abs{x_i} + \abs{y_i})^{e_i}\qquad \text{for all } x_i,y_i \in \Z.\label{eq:qbound}
	\end{align}
	Now, we apply a substitution $\phi$ defined by $\phi(x_i) = \kappa_{x_i} z^{\nu(a_i)}$ and $\phi(y_i) = \kappa_{y_i} z^{\nu(a_i)}$ for all $i$ where $z$ is a new variabe and $ \kappa_{x_i}$, $ \kappa_{y_i} \in \N$ are constants (defined below). We apply $\phi$ on both sides of (\ref{eq:qbound})~-- clearly this preserves the inequality.
	Every monomial $\omega =  x_1^{e_1} \cdots x_n^{e_n} \cdot y_1^{f_1} \cdots y_n^{f_n}$ then becomes $\phi(\omega) =  \kappa_{x_1}^{e_1} \cdots  \kappa_{x_n}^{e_n} \cdot  \kappa_{y_1}^{f_1} \cdots  \kappa_{y_n}^{f_n}\cdot z^{\nu(\omega)}$.
	
	It remains to show that if a monomial $\omega$ of $q_k$ with $\nu(\omega) = \nu$ has non-zero coefficient, then the monomial $z^\nu$ has non-zero coefficient in $\phi(q_k)$.
	Once we have established that, it follows that every monomial $\omega$ of $q_k$ satisfies $\nu(\omega) \leq \nu(a_k)$, because the substituted polynomial on the right side of (\ref{eq:qbound}) is of degree $\nu(a_k)$ and, thus, by (\ref{eq:qbound}) also $\phi(q_k)$ is of degree at most  $\nu(a_k)$. This gives the desired bound on the weights of the multiplication polynomials $q_1, \dots, q_n $.
	
	Now, let $M$ be a bound on the absolute values of all coefficients of $q_k$ and $N$ the number of monomials of $q_k$.
	Let $B\in \N$ such that $B > c$ (the nilpotency class) and $2^B > MN$. We set $\kappa_{x_i} = 2^{B^{i}}$ and $\kappa_{y_i} = 2^{B^{n + i}}$. 
	Recall that by \cite[7.4]{LeedhamGreenS98}, we have $e_i, f_i \leq c$ for every monomial $\omega =  x_1^{e_1} \cdots x_n^{e_n} \cdot y_1^{f_1} \cdots y_n^{f_n}$ of $q_k$.

    For $(e_1, \dots, e_n, f_1, \dots, f_n) \neq (e'_1, \dots, e'_n, f'_1, \dots, f'_n)$ with $0\leq e_i, f_i, e'_i, f'_i \leq c$,
     we have
    $$\abs{\sum_i\left( e_iB^{i} + f_i B ^{n+i}\right) - \sum_i\left( e'_iB^{i} + f'_i B ^{n+i}\right)} \geq B.$$
    Therefore, under the same assumption, the products 
    $\kappa_{x_1}^{e_1} \cdots  \kappa_{x_n}^{e_n} \cdot  \kappa_{y_1}^{f_1} \cdots  \kappa_{y_n}^{f_n} = 2^{\sum_i\left( e_iB^{i} + f_i B ^{n+i}\right)}$ and $\kappa_{x_1}^{e'_1} \cdots  \kappa_{x_n}^{e'_n} \cdot  \kappa_{y_1}^{f'_1} \cdots  \kappa_{y_n}^{f'_n} = 2^{\sum_i\left( e'_iB^{i} + f'_i B ^{n+i}\right)}$ differ by at least a factor $2^B$. 
    
    Now, assume that some monomial gets cancelled by the substitution; that means, in particular,
    $$ \sum_{j=1}^{N'} \alpha_j \kappa_{x_1}^{e_{j,1}} \cdots  \kappa_{x_n}^{e_{j,n}} \cdot  \kappa_{y_1}^{f_{j,1}} \cdots  \kappa_{y_n}^{f_{j,n}} = 0$$
    for some $1\leq N' \leq N$ and non-zero coefficients $\alpha_j\in \Z$ with $\abs{\alpha_j} \leq M$ and pairwise distinct exponent vectors $(e_{j,1}, \dots  ,e_{j,n},f_{j,1}, \dots  ,f_{j,n})$ for $j=1,\dots, N$. W.\,l.\,o.\,g.\ let $\kappa_{x_1}^{e_{1,1}} \cdots  \kappa_{x_n}^{e_{1,n}} \cdot  \kappa_{y_1}^{f_{1,1}} \cdots  \kappa_{y_n}^{f_{1,n}}$ be the maximal $\kappa_{x_1}^{e_{j,1}} \cdots  \kappa_{x_n}^{e_{j,n}} \cdot  \kappa_{y_1}^{f_{j,1}} \cdots  \kappa_{y_n}^{f_{j,n}}$. Then we have
    \begin{align*}
    \abs{ \alpha_1 \kappa_{x_1}^{e_{1,1}\cdot  \kappa_{y_1}^{f_{1,1}} \cdots  \kappa_{y_n}^{f_{1,n}}} \cdots  \kappa_{x_n}^{e_{1,n}}} &= \abs{\sum_{j=2}^N \alpha_j \kappa_{x_1}^{e_{j,1}} \cdots  \kappa_{x_n}^{e_{j,n}} \cdot  \kappa_{y_1}^{f_{j,1}} \cdots  \kappa_{y_n}^{f_{j,n}}}\\
    &\leq N M \max_{j= 2, \dots, N} \abs{\kappa_{x_1}^{e_{j,1}} \cdots  \kappa_{x_n}^{e_{j,n}} \cdot  \kappa_{y_1}^{f_{j,1}} \cdots  \kappa_{y_n}^{f_{j,n}}}\\
    &\leq \frac{NM}{2^B}\kappa_{x_1}^{e_{1,1}} \cdots  \kappa_{x_n}^{e_{1,n}}\cdot  \kappa_{y_1}^{f_{1,1}} \cdots  \kappa_{y_n}^{f_{1,n}}\\
    &< \kappa_{x_1}^{e_{1,1}} \cdots  \kappa_{x_n}^{e_{1,n}}\cdot  \kappa_{y_1}^{f_{1,1}} \cdots  \kappa_{y_n}^{f_{1,n}}.
    \end{align*}
    This is a contradiction since $\abs{\alpha_1} \geq 1$.
    Therefore, the monomials of weight $\nu$ of $q_k$ do not all cancel to zero after applying $\phi$. Hence, the monomial $z^\nu$ has non-zero coefficient in $\phi(q_k)$.
\end{proof}

\begin{thm}\label{thm:upperbound}
	Let $k$ denote the rank of $G/[G,G]$, and $c$ the nilpotency class of $G$. Then the matrix representation produced by Nickel's algorithm has dimension at most $\sum_{i=0}^{c-1}k^i + \rank(\Gamma_c(G)) < 2k^c$. Moreover, it has never larger dimension than Jennings' embedding \cite{Jennings55}.
\end{thm}
For fixed nilpotency class Theorem~\ref{thm:upperbound} provides a polynomial bound on the dimension on the embedding. Since not only the dimension, but also the number of occurring monomials is bounded polynomially, we obtain also a polynomial running time of Nickel's algorithm for bounded nilpotency class (the \emph{Insert} routine in  Algorithm~\ref{alg:Nickel} can be called only at most as many times as there are different monomials~-- each call to \emph{Insert} needs time polynomial in the number of bits required to represent the polynomials already in the basis and the newly inserted polynomial).
\begin{proof}
	By definition, the $G$-module produced by Nickel's algorithm is the span of the multiplication polynomials (\ref{multiplicationpolynomials}) where the variables $y_1, \dots, y_n$ are substituted by integer values. Like in the proof of Theorem \ref{thm:upperboundUT}, this $G$-module is contained in the span of the monomials of all the multiplication polynomials (again the variables $y_1, \dots, y_n$ are substituted by integer values). As every non-linear monomial of some $q_i$ has at least one variable from $y_1, \dots, y_n$, we know by Lemma~\ref{lem:weightbound} that for any non-linear monomial $\omega = x_1^{e_1} \cdots x_n^{e_n}$ occurring at any time during Nickel's algorithm, we have $\nu (\omega) = \sum e_i\nu(a_i) \leq c - 1$.
	
	Now, compare this to Jennings' embedding as described in Section~\ref{sec:Jennings}: for every $M \in \N$ there is a canonical one-to-one correspondence between basis elements of $Z[G]/I^{M+1}$ and terms $x_1^{e_1} \cdots x_n^{e_n}$ with $\sum e_i\nu(a_i) \leq M$ (again $I$ is the augmentation ideal). Now, in \cite[Prop.\ 3.3]{LoO99} the dimension of $\Q[G]/I^{M+1}$ over $\Q$ was computed to be at most $\sum_{i=0}^Mk^i$. 
	Since in our case all non-linear monomials have weight at most $c-1$, we obtain an upper bound of $\sum_{i=0}^{c-1}k^i$ for the number of monomials of weight at most $c-1$. As we have not counted the linear monomials $x_i$ for $\nu(a_i)=c$, we have to add the rank of $\Gamma_c(G)$ as a free abelian group.
	In order to see that the dimension is at most as large as Jennings' embedding just observe that all (non-linear and linear) monomials have weight at most $c$. This yields also the bound $2k^c$.
\end{proof}

\section{Nickel's and Jennings' embedding in other classes of groups}\label{NickJen}
	
\subsection{Heisenberg groups}
	As before, let $e_{i,j}(\alpha)$ for $i<j$ be the matrix with $ij$-th entry $\alpha$ and the rest of the entries $0$, and let $s_{i,j}(\alpha)=1+e_{i,j}(\alpha)$ and $s_{i,j}=s_{i,j}(1)$. 
    The $(2m+1)$-dimensional \emph{Heisenberg group} is defined as
	$$G= \gen{a_1, \dots , a_{2m+1}} \leq \UTZ{m+2}  $$
	where 
	$$a_i = \begin{cases} 
	s_{1,i+1} &\mbox{for}\:1 \leq i\leq m, \\
	s_{i-m+1,m+2}&\mbox{for}\:m+1 \leq i < 2m+1, \\
	s_{1,m+2} &\mbox{for}\:i=2m+1.
	\end{cases}$$

	Using the facts that 
	$$s_{i,j}^{-1}=\left(s_{i,j}(1)\right)^{-1}= s_{i,j}(-1),$$  $$[s_{i,j}, s_{j,k}]=s_{i,k}, \: \text{and}\: [s_{j,i}, s_{k,j}]=s_{i,k}(-1) \: \text{for} \: i<j<k,$$ 
	we  can give a finite presentation of the $(2m+1)$-dimensional Heisenberg group:
		$$G= \genr{a_1, \dots , a_{2m+1}}{R}  $$
	with
	$$R=\{ [a_i,a_{m+i}]=a_{2m+1}\:\text{for} \:1 \leq i\leq m \:\text{and all other pairs of}\:a_j\: \text{commute}\}.$$
	Moreover, $(a_1,a_2,\dots, a_{2m+1})$ is a Mal'cev Basis for the Heisenberg group $G$.

\begin{thm}\label{thm:Heisenberg group} Let $G$ be the $(2m+1)$-dimensional Heisenberg group. The size of Jennings' embedding of $G$ is $2m^2+3m+2$ and the size of Nickel's embedding it is $2m+2$. In particular, the size of the matrix obtained in Jennings' embedding is larger than that of Nickel's. 
\end{thm} 
\begin{proof}
	We will first have a look at the size of the matrices of the image of the Heisenberg group under Nickel's embedding. For simplicity we will write $ a_1^{x_1}a_2^{x_2} \cdots a_{2m+1}^{x_{2m+1}}=\vec{a}^{\vec{x}}$. 
	For $ 1\leq j \leq m $, we have
	$$t_i^{a_j^{-k}}(\vec{a}^{\vec{x}})= \begin{cases}{x_{j}-k} &\mbox{for}\: i=j \\ x_i &\mbox{for}\:i\neq j\:\text{and}\: i\neq 2m+1 \\ x_{2m+1}+kx_{m+j} &\mbox{for}\:i=2m+1 \end{cases} $$
	For $ m+1 \leq j \leq 2m+1 $,
	
	$$t_i^{a_j^{-k}}(\vec{a}^{\vec{x}})= \begin{cases}{x_{j}-k} &\mbox{for}\: i=j \\ x_i &\mbox{for}\:i\neq j \end{cases} $$
	
	In order to see this, we only need to have a look at
	$$ \vec{a}^{\vec{x}}a_j^{-k}=a_1^{x_1} \cdot a_2^{x_2}\cdots a_{2n+1}^{x_{2m+1}}a_j^{-k}.$$
	
	For $ 1\leq j \leq m $, since $s_{1,j+1}=a_j$ commutes with all except $s_{j+1,m+2}=a_{m+j}$ and also $[s_{1,j+1}^{-k},s_{j+1,m+2}^{-y_{m+j}}]=s_{1,m+2}^{ky_{m+j}}=a_{2m+1}^{ky_{m+1}}$, and, moreover, $a_{2m+1}$ commutes with everything, we have   
	\begin{align*}
		\vec{a}^{\vec{x}}a_j^{-k}&= s_{1,2}^{x_1}s_{1,3}^{x_2} \cdots s_{1,j+1}^{x_j} \cdots s_{2,m+2}^{x_{m+1}} \cdots s_{1,m+2}^{x_{2m+1}}\cdot s_{1,j+1}^{-k} 
		\\&=s_{1,2}^{y_1}s_{1,3}^{x_2} \cdots s_{1,j+1}^{x_j-k} \cdots s_{2,m+2}^{x_{m+1}} \cdots s_{1,m+2}^{x_{2m+1}+kx_{m+j}}.
	\end{align*}
	Since $s_{j-m+1,m+2}^{x_j}$ for $ m+1\leq j \leq 2m $ commutes with all other elements which are positioned on the rightside in Mal'cev basis, it follows
	\begin{align*}
		\vec{a}^{\vec{x}}a_j^{-k}&= s_{1,2}^{x_1}s_{1,3}^{x_2} \cdots s_{1,m+1}^{x_m} \cdots s_{1,j+1}^{x_j} \cdots s_{1,m+2}^{x_{2m+1}} \cdot s_{j-m+1,m+2}^{-k} \\&=s_{1,2}^{x_1}s_{1,3}^{x_2} \cdots s_{1,m+1}^{x_m} \cdots s_{j-m+1,j+1}^{x_j-k} \cdots s_{1,m+2}^{x_{2m+1}}.\end{align*}
	Finally, for $j=2m+1$, since $a_{2m+1}=s_{1,m+1}$ is the last element in the product, we have 
	$$\vec{a}^{\vec{x}} a_{2m+1}^{-k}= s_{1,2}^{y_1}s_{1,3}^{y_2} \cdots s_{1,m+1}^{y_m}  \cdots s_{1,m+2}^{y_{2m+1}-k}.$$
	
	In conclusion, we have $\{t_1,t_2, \dots, t_{2m+1}, 1\}$ as the $\Q$-basis for the $G$-module. Hence, the size of the matrices  under Nickel's embedding is $2m+2$. 
	
    Now let us compute the size of the matrices obtained under Jennings' embedding. 
	First set $u_i= 1-a_i$ for $1\leq i \leq m $ and $v_j= 1-a_{m+j}$ for $1 \leq j \leq m$ and $w= 1-a_{2m+1}$. Notice that $\nu(u_i)=\nu(v_i)=1$ for all $1\leq i \leq m $ and $\nu(w)=2$.
	Hence, the elements of the basis for $\Q G / I^3$ are of the following forms
	\begin{itemize}
	\item $1$, $w$,
	\item  $u_i, v_i,\: u_i^2, v_i^2$ for $1\leq i\leq m$,
	\item  $u_iu_j, \: v_iv_j$ for $ 1\leq i< j \leq m$ or $u_iv_j$ for $ 1\leq i, j \leq m$.
	\end{itemize}  The number of elements of the forms $u_i, v_i$ and $u_i^2, v_i^2$ is $4m$ and  the number of elements of the forms $u_iu_j$, $v_iv_j$ and $u_iv_j$ is $\binom{2m}{2}$. As a result, the total number of basis elements is $$2+4m+\binom{2m}{2}= 2m^2+3m+2.$$
\end{proof}
	
\subsection{Free nilpotent groups}
	Let $F(k,c)$ denote the free nilpotent group with $k$ generators and nilpotency class $c$.
	For Nickel's embedding we have an obvious lower bound: the Hirsch length. By Witt's formula (see e.\,g.\ \cite[Thm.\ 5.7]{Hall57} or \cite[Thm.\ 5.11]{MKS66}), the Hirsch length is $\frac{1}{c} k^c + \Oh(\frac{1}{c} k^{c-1})$. Theorem~\ref{thm:upperbound} yields the upper bound of $\rank(\Gamma_c(G)) + \sum_{i = 0}^{c-1} k^i $, which again by Witt's formula  is bounded by $\frac{1}{c} k^c + \sum_{i = 0}^{c-1} k^i +  \Oh(k^{c/2}) = \frac{1}{c} k^c +  k^{c-1} +  \Oh(k^{c-2})$. Thus, lower and upper bound lie only by a factor $ 1 + \frac{c}{k}$ apart (plus lower order terms).
	
	On the other hand by \cite[Prop.\ 6.1]{LoO99}, we know that Jennings' embedding of $F(k,c)$ has dimension exactly $\sum_{i = 0}^c k^i$.

	\subsection{Direct and central products}
	
	Let $G$ and $H$ be two arbitrary $\tau$-groups with Mal'cev bases $\vec a' = (a_1, \dots, a_m)$ and $\vec a'' = (a_{m + 1}, \dots, a_n)$ respectively. Then  $\vec a = (a_1, \dots, a_m, a_{m + 1}, \dots, a_n)$ is a Mal'cev basis of $G \times H$. 
	
	Let $t_i\in \Z[x_1, \dots, x_n]$ with $1 \leq i \leq m$ be one of the coordinate functions defined by $\vec a$. Then $t_i^h = t_i$ for every $h \in H$~-- and thus also $t_i^{gh} = t_i^g$ for every $g\in G$ and $h \in H$ because $g$ and $h$ commute. Moreover, $ t_i^g$ obviously does not depend on variables $x_{m + 1}, \dots, x_n$.
	Likewise for $m + 1 \leq i \leq n$, we have $t_i^{gh} = t_i^h$ for every $g \in G$ and $h \in H$ and $t_i^h$ does not depend on variables $x_{1}, \dots, x_m$.
	Thus, if $Q' \subseteq \Z[x_1, \dots, x_m]$ and $Q''\subseteq \Z[x_{m+1}, \dots, x_n]$ are the bases of $G$ and $H$ computed by Nickel's algorithm, then $Q = Q' \cup Q'' \subseteq \Z[x_1, \dots, x_n]$ is the basis for $G\times H$ which is computed by Nickel's algorithm. Moreover, $Q' \cap Q'' = \oneset{1}$ because the sets of variables occurring in $Q'$ and $Q''$ are disjoint (obviously, the constant polynomial is contained in both $Q'$ and $Q''$).
	Thus, we have the following:
	
\begin{prop}
	Let $M$ (resp.\ $N$) be the dimension of Nickel's embedding of $G$ (resp.\ $H$) into $\UTZ{M}$ (resp.\ $\UTZ{N}$). Then the Nickel's embedding of $G\times H$ has dimension $M + N - 1$.
\end{prop}

	Let us consider a slight generalization of the direct product: a very special type of the central product.
	The \emph{central product} $G \times_C H$  of $G$ and $H$ is defined as $G\times H / \oneset{a_m = a_n}$. Of course, this depends on the Mal'cev bases chosen for $G$ and $H$. A Mal'cev basis for the central product is $(a_1, \dots, a_{m-1}, a_{m+1}, \dots, a_n)$.
	
	For $t_i^{gh}$ with $m\neq i\neq n$ the above considerations for the direct product hold also for the central product. It only remains to look at $t_n^{gh}(x_1, \dots, x_n)$. If $g\in G$ is of the form $a_1^{k_1} \cdots a_{m-1}^{k_{m-1}}$, we have $t_n^{g}(x_1, \dots, x_n) 
	= x_n + p_g(x_1, \dots, x_{m-1})$ for some polynomial $p_g$. 
	
	Also, if $h\in H$ is of the form $a_{m+1}^{k_{m+1}} \cdots a_{n-1}^{k_{n-1}}$, we have $t_n^{h}(x_1, \dots, x_n) 
	= x_n + p_h(x_{m+1}, \dots, x_{n-1})$ for some polynomial $p_h$. 
	Moreover, $t_n^{a_n^k}(x_1, \dots, x_n) = x_n +k$ for $k \in \Z$.
	An arbitrary element of $G \times_C H$ can be written as $gha_n^k$ where $g$ is of the form $a_1^{k_1} \cdots a_{m-1}^{k_{m-1}}$ and $h$ of the form $a_{m+1}^{k_{m+1}} \cdots a_{n-1}^{k_{n-1}}$.
	Thus, we have $t_n^{gha_n^k}(x_1, \dots, x_n) = x_n + p_g(x_1, \dots, x_{m-1}) + p_h(x_{m+1}, \dots, x_{n-1}) + k$.
	
	Hence, if $Q'$ and $Q''$ are the bases of $G$ and $H$ computed by Nickel's algorithm, then $Q =( Q' \cup Q'')/\oneset{x_m = x_n}$ is the basis for $G\times H$ which is computed by Nickel's algorithm. Again, $Q' \cap Q'' = \oneset{1}$, what leads to the following:
	
\begin{prop}
	Let $M$ (resp.\ $N$) be the dimension of Nickel's embedding of $G$ (resp.\ $H$) into $\UTZ{M}$ (resp.\ $\UTZ{N}$). Then the Nickel's embedding of the central product $G\times_C H$ has dimension $M + N - 2$.
\end{prop}

	Now, let us take a look at Jennings' embedding. The following example shows that the dimension of the embedding of $G \times H$ is not even bounded by the product of the dimensions of the embeddings of $G$ and $H$:

\begin{example}
	Let $G = \Z^k = \gen{e_1, \dots, e_k}$ for some $k \in \N$ and let $H =\Z^c \rtimes_\phi \Z $ where the action $\Z = \gen{a}$ on $\Z^c= \gen{f_1, \dots, f_c}$ is defined by $\phi(f_c) = f_c$ and $\phi(f_i )= f_i f_{i+1}$ for $1 \leq i \leq c-1$ (we use multiplicative notation). Then $G$ is nilpotent of class $1$ and $H$ is nilpotent of class $c$. 
	
	The basis of the $G$-module produced by Jennings' embedding is $(1, 1 - e_1, \dots, 1 - e_k)$. For $H$ the basis is more complicated. Set $u_0 = 1-a$ and $u_i = 1 - f_i$ for $1 \leq i \leq c$. 
	Then, the set $\Set{ \prod_{i=0}^{c} u_i^{r_i}}{r_0 + \sum_{i = 1}^{c}i r_i \leq c}$ forms a basis of the the $H$-module produced by Jennings' embedding. Now, every tuple $(r_1, \dots, r_c)$ with $\sum_{i = 1}^{c}i r_i = c$ defines a partition of the number $c$ (i.\,e.\ a way how to write $c$ as a sum of natural numbers). It is well-known that the number of different partitions of some number $c$ is bounded by $2^{\gamma \sqrt{c}}$ for some constant $\gamma$ (see e.\,g.\ \cite{Pribitkin09}). Therefore, the total number of basis elements for the embedding of $H$ is bounded by $c^2 2^{\gamma \sqrt{c}}$.

	Now, let us look at the basis produced for $G \times H$. This basis, in particular, contains all elements of the form $\Set{ \prod_{i=0}^{k} v_i^{r_i}}{\sum_{i = 0}^{k} r_i \leq c}$ where $v_0 = 1-a$ and $v_i = 1- e_i$ for  $1 \leq i \leq k$. These terms can be identified with the set of polynomials in $k+1$ variables and degree at most $c$: thus, there are at least $\binom{k + c}{c}$ many of them. If we assume that $k = c$, this means that the resulting basis has at least approximately $4^k = 4^c$ elements~-- a huge blow-up compared to $k+1$ and $2^{\gamma \sqrt{c}}$.
	
	Note that this construction also works for central products: we simply choose $G= \Z^{k+1}$ and identify the last basis element with the central generator $f_c$ of $H$.
\end{example}

\section{Open Questions}\label{OpenProblems}
	
	We have seen that, in general, the size of the output of Nickel's algorithm does not depend polynomially on the Hirsch length of the input. However, by reordering the Mal'cev basis of unitriangular matrices, we could obtain a polynomial bound on the dimension of the matrix representation. Thus, the following remains open:
\begin{itemize}
	\item What are tight upper and lower bounds on the dimension of Nickel's embedding for $\UTZ{m}$?
	\item Does every $\tau$-group have a Mal'cev basis such that Nickel's algorithm produces a matrix representation of polynomial size?
\end{itemize}
	We conjecture that the answer is 'no'.
    If this conjecture is true, the following more general question remains open:
\begin{itemize}
	\item Does every $\tau$-group allow a matrix representation of polynomial (in the Hirsch length) size? What is the minimal bound in terms of Hirsch length and nilpotency class?
\end{itemize}
    Independently of whether the answer is 'yes' or 'no', more precise lower and upper bounds on minimal matrix representations of $\tau$-groups would be of great interest.
    
    Another open question is the time complexity of Nickel's algorithm. We have seen, that it is not polynomial in the input size. Still the following question is of interest:
\begin{itemize}
    \item Is the running time of Nickel's algorithm polynomial in the dimension of the matrix representation?
\end{itemize}
The answer to this question is not obvious: Although the size of dimension of the embedding might be of polynomial size, still the number of monomials appearing during the computations might be exponential. In this case the running time also is exponential. There is no obvious reason why such a situation should not occur.

If the nilpotency class is fixed, the dimension of the matrix representation is polynomial by Theorem~\ref{thm:upperbound}.
\begin{itemize}
    \item For fixed nilpotency class give a precise (polynomial) bound on the running time.
\end{itemize}

\bibliography{nilpotent}

\end{document}